\documentclass[11pt]{amsart}
\usepackage{amssymb,amsmath,amsfonts,amsthm,mathrsfs,amscd,graphicx,color}
\usepackage{bibunits}
\usepackage[all]{xy}
\usepackage{multicol}
\usepackage[colorlinks]{hyperref}
\usepackage{fullpage}
\usepackage{mathpazo}
\usepackage{hyperref}

\defaultbibliography{biblio}
\defaultbibliographystyle{plain}

\def\a{\alpha}
\def\ovX{\bar{X}}
\def\ovY{\bar{Y}}

\DeclareMathOperator{\alb}{Alb}
\DeclareMathOperator{\aut}{Aut}
\DeclareMathOperator{\CH}{CH}
\DeclareMathOperator{\ch}{ch}

\DeclareMathOperator{\cl}{cl}
\DeclareMathOperator{\Fq}{\mathbb{F}_q}
\DeclareMathOperator{\id}{id}

\DeclareMathOperator{\Pic}{Pic}

\DeclareMathOperator{\td}{td}

\DeclareMathOperator{\Tr}{Tr}

\newcommand{\bF}{\overline{\mathbb{F}}}
\newcommand{\bX}{\bar{X}}

\newcommand{\C}{\mathbb{C}}
\newcommand{\E}{\mathcal{E}}
\newcommand{\F}{\mathbb{F}}

\newcommand{\Ql}{\mathbb{Q}_{\ell}}
\def\et{\text{\'et}}

\def\rat{\mathbb{Q}}

\newif\ifHideFoot
\HideFoottrue % This line can be left alone.
\HideFootfalse % Comment this line to hide footnotes\\

\if0
\newcommand{\marg}[1]{\normalsize{{\color{red}\footnote{{\color{blue}#1}}}{\marginpar[{\vskip
    -.3 cm\color{red}\hfill\tiny\thefootnote$\rightarrow$}]{{\vskip -.3 cm
     \color{red}$\leftarrow$\tiny\thefootnote}}}}}
\newcommand{\Yano}[1]{\marg{(Yano) #1}}
\newcommand{\Jeff}[1]{\marg{(Jeff) #1}}

\newcommand{\Charles}[1]{\marg{(Charles) #1}}
\fi

\ifHideFoot
\renewcommand{\Yano}[1]{}
\renewcommand{\Jeff}[1]{}
\renewcommand{\Charles}[1]{}
\fi

\newtheorem{theorem}{Theorem}[section]
\newtheorem*{thmA}{Theorem A}
\newtheorem*{thmB}{Theorem B}

\newtheorem{lemma}[theorem]{Lemma}

\theoremstyle{remark}

\newtheorem{definition}[theorem]{Definition}
\newtheorem{empt}[theorem]{\!\!}

\newtheorem{teo}{Theorem}[section]

\newtheorem{lem}[teo]{Lemma}
\newtheorem{cor}[teo]{Corollary}

\newtheorem*{teoalpha*}{Theorem}
\newtheorem*{coralpha*}{Corollary}
\newtheorem*{conalpha*}{Conjecture}

\theoremstyle{definition}

\theoremstyle{remark}

\DeclareMathOperator{\derived}{D}

\def\red{{\rm red}}

\def\et{{\rm \acute et}}

\def\inv{^{-1}}
\def\udot{^{\bullet}}

\def\rat{{\mathbb Q}}

\def\calo{{\mathcal O}}

\def\iso{\cong}

\renewcommand{\bar}[1]{{\overline{#1}}}

\DeclareMathOperator{\gal}{Gal}

\DeclareMathOperator{\pic}{Pic}

\DeclareMathOperator{\G}{G}
\global\let\P\undefined
\DeclareMathOperator{\P}{P}

\def\tensor{\otimes}

\newenvironment{alphabetize}{\begin{enumerate}

}{\end{enumerate}}

\let\lra\longrightarrow
\let\xra\xrightarrow

\makeatletter
\let\@wraptoccontribs\wraptoccontribs
\makeatother

\title{Derived equivalence, Albanese varieties, \\ and the zeta functions of $3$--dimensional varieties}

\author{Katrina Honigs}
\address{University of Utah, Mathematics Department, Salt Lake City, UT 84112}
\email{honigs@math.utah.edu}
\thanks{The author is partially supported by a Mathematical Sciences Postdoctoral Research Fellowship, Grant No.\ 1606268.}

\contrib[with an appendix by]{Jeffrey D. Achter}
 \address{Colorado State University, Department of Mathematics,
  Fort Collins, CO 80523,
  USA}
 \email{j.achter@colostate.edu}
\contrib[]{Katrina Honigs}
\contrib[]{Sebastian Casalaina-Martin}
 \address{University of Colorado, Department of Mathematics, 
  Boulder, CO 80309, USA }
 \email{casa@math.colorado.edu}
\contrib[]{Charles Vial}
\address{Fakult\"at f\"ur Mathematik, 
Universit\"at Bielefeld, P.O.Box 100 131, D-33 501 Bielefeld, Germany}
\email{c.vial@dpmms.cam.ac.uk}

\date{\today}

\begin{document}
\begin{bibunit}

\begin{abstract}
We show that any 
derived equivalent smooth, projective varieties of dimension $3$
over a finite field $\F_q$ have equal zeta functions.
This result is an application of the extension to smooth, projective varieties over any field of 
Popa and Schnell's proof that 
derived equivalent smooth, projective varieties over $\mathbb{C}$ have isogenous Albanese torsors; this result is proven in an appendix by Achter, Casalaina-Martin, Honigs and Vial.
\end{abstract}

\maketitle

The problem of characterizing the bounded derived category of coherent sheaves of a variety has connections to birational geometry, the minimal model program, mirror symmetry (in particular, the conjecture of Kontsevich \cite{kontsevich}), 
and motivic questions. 

Orlov has conjectured that derived equivalent smooth, projective varieties have isomorphic motives \cite{motives}. This conjecture predicts that smooth, projective
varieties over a finite field 
that are derived equivalent 
have equal zeta functions. 
The prediction holds in the case of curves since derived equivalent
smooth, projective curves over a finite field are isomorphic:
proof in the genus 1 case is given by
Antieau, Krashen and Ward \cite[Example~2.8]{ward}, and proof in all other cases is a consequence of 
Bondal and Orlov \cite[Theorem~2.5]{ample}, which shows that 
derived equivalent varieties with ample or anti-ample canonical bundle must be isomorphic.
In \cite{kh},
it was verified   that derived equivalent smooth, projective varieties 
over a finite field
that are abelian or of dimension 2  have equal zeta functions.

In this paper, we prove the following extension to these results:

\begin{thmA}\hypertarget{dimthree}{}
Let $X,Y/\F_q$ be derived equivalent smooth, projective varieties of dimension  $3$, where
$\F_q$ is a finite field with $q$ elements. Then 
$\zeta(X)=\zeta(Y)$.
\end{thmA}

The proof of Theorem~A is similar to the argument in \cite{kh} 
proving that derived equivalent smooth, projective surfaces over any finite field have equal zeta functions:  it is accomplished
by  comparing the 
eigenvalues of the geometric Frobenius morphism acting on the
$\ell$-adic \'etale cohomology groups of the varieties in question. 

The crucial ingredient for making this comparison between the point-counts of
 three-dimen{\-}sional varieties is
 the following theorem,
proven in Appendix~\ref{app}, 
 which has as a corollary that
if $X$ and $Y$ are derived equivalent smooth, projective varieties over a finite field $\Fq$ and $\ovX$, $\ovY$ are their base changes to $\bar{\Fq}$, then there is an isomorphism
\[H^1_{\et}(\ovX,\Ql)\cong H^1_{\et}(\ovY,\Ql)\]
that is compatible with the action of 
the $q$-th power geometric Frobenius morphism:

\begin{thmB}\hypertarget{alb}{}
Derived equivalent smooth, projective varieties $X$ and $Y$ over an arbitrary field $k$
have isogenous Albanese varieties.
\end{thmB}

The proof of Theorem~B
is given by Popa and Schnell's proof 
that if $X,Y/\mathbb{C}$ are derived equivalent varieties,
then $\Pic^0(X)$ and $\Pic^0(Y)$ must be isogenous
\cite{popa}, with a few small changes.

An alternate proof of  \hyperlink{alb}{Theorem~B}
over $\C$ has also been obtained by R.\ Abuaf in his Theorem~3.0.14 of \cite{abuaf}.  It is conceivable that similar methods to those in loc.\ cit.\ can be used over algebraically closed fields of arbitrary characteristic.

%In Section \ref{sec.on}, we provide some background. 
%In Section~\ref{sec.b}, we prove \hyperlink{alb}{Theorem~B}. In Section~\ref{sec.zeta}, we prove \hyperlink{dimthree}{Theorem~A}.

\subsubsection*{Acknowledgements}
Thanks to Aaron Bertram, Christopher Hacon, Daniel Litt, Luigi Lombardi, 
Martin Olsson,
Mihnea Popa, 
Rapha\"{e}l Rouquier, Christian Schnell and Sofia Tirabassi for helpful comments, conversations and correspondence.

%The author is partially supported by a Mathematical Sciences Postdoctoral Research Fellowship, Grant No.\ 1606268.

\section{Background}

We take a \textit{variety} to be a separated, integral scheme of finite type over a field.
In this section, $X$ and $Y$ denote smooth, projective varieties.

\begin{definition}
An exact functor $F$ between derived categories $D^b(X)$ and $D^b(Y)$ is a \textit{Fourier--Mukai transform} if there exists an object $P \in D^b(X\times Y)$, called a \textit{Fourier--Mukai kernel}, such that 
\begin{equation}\label{fmdef}
F\cong {p_Y}_*(p_X^*(-)\otimes P)=:\Phi_P,
\end{equation}
where $p_X$ and $p_Y$ are the projections $X\times Y\to X$ and and $X\times Y\to Y$.
A Fourier--Mukai transform that is an equivalence of categories is called a \textit{Fourier--Mukai equivalence}.
The pushforward, pullback, and tensor 
in \eqref{fmdef}
are all in their derived versions, but the notation
is suppressed. 

A \textit{derived equivalence} is an exact equivalence between derived categories; varieties are said to be derived equivalent if their associated bounded derived categories are. By the following theorem, in the context of this paper, derived equivalence and Fourier--Mukai equivalence are synonymous.
\end{definition}

\begin{theorem}[Orlov {{\cite[Theorem 3.2.1]{Ocoh}}}]\label{orlov.kernel.exist}
Let $X$ and $Y$ be smooth projective varieties and
 $F:D^b(X)\to  D^b(Y)$  an exact equivalence.  
Then there is an object $\mathcal{E}\in D^b(X\times Y)$ such that $F$ is isomorphic to the functor $\Phi_{\mathcal{E}}$, and the object $\mathcal{E}$ is determined uniquely up to isomorphism. 
\end{theorem}

The full statement of \cite[Theorem 3.2.1]{Ocoh} is stronger than what is given here, but the statement in 
Theorem~\ref{orlov.kernel.exist} is sufficient for the purposes of this paper.

\begin{empt}
Let $X$ and $Y$ be smooth, projective varieties over a perfect field.
Any Fourier--Mukai transform gives a map on Chow groups:
The functor $\Phi_{\E}$ induces a map
$$\Phi_{\E}^{\CH}=p_{Y*}(v(\E)\cup p_X^*(-)): \CH(X)_{\mathbb{Q}}\to\CH(Y)_{\mathbb{Q}}$$
where $v(\E):= \ch (\E).\sqrt{\td(X\times Y)}$ is the \textit{Mukai vector} of $\E$ (see for instance \cite[Definition 5.28]{huybrechts}). 
Since $\Phi_{\E}$ is an equivalence, $\Phi_{\E}$ is a bijection (cf.\ \cite[Remark 5.25, Proposition 5.33]{huybrechts}).

Similarly, the cycle class of $v(\E)$ inside any Weil cohomology theory $H$ 
(i.e., de~Rham, singular, crystalline or $\ell$-adic \'etale)
induces a map $\Phi_{\E}^H=p_{Y*}(\cl(v(\E))\cup p_X^*(-))$ that factors through the above map on Chow rings with rational coefficients. 
This map on cohomology does not necessarily preserve degree, and,  
in the case of $\ell$-adic \'etale cohomology of varieties over finite fields,
Tate twists must be accounted for the map to be compatible with the of action geometric Frobenius, so care must be taken with the domain and codomain of $\Phi_{\E}^H$.
The map $\Phi_{\E}^H$ gives the following isomorphisms compatible with the action of geometric Frobenius $\varphi$ between the \textit{even} and \textit{odd Mukai--Hodge structures} \cite{kh,thesis}, of $X$ and $Y$, where $d$ denotes $\dim(X)(=\dim(Y))$:
\begin{align}
\bigoplus^{d}_{i=0} H^{2i}({X})(i)&\cong \bigoplus^{d}_{i=0} H^{2i}({Y})(i),
\label{mhon}
\\
\bigoplus^{d}_{i=1} H^{2i-1}({X})(i)&\cong \bigoplus^{d}_{i=1} H^{2i-1}({Y})(i).\label{mhtw}
\end{align}
\end{empt}

\section{Zeta Functions}\label{sec.zeta}

\begin{thmA}
Let $X,Y/\F_q$ be derived equivalent smooth, projective varieties of dimension  $3$, where
$\F_q$ is a finite field with $q$ elements. Then 
$\zeta(X)=\zeta(Y)$.
\end{thmA}

\begin{proof}
Let $\ovX$, $\ovY$ be the base changes of $X$ and $Y$ to the algebraic closure 
$\bF_q$ of $\F_q$.
Fix $\ell\in\mathbb{Z}^+$ prime such that $(q,\ell)=1$. 

By the Lefschetz fixed-point formula for Weil cohomologies (see Proposition~1.3.6 and Section~4 of 
Kleiman \cite{kleiman}), to prove this theorem it is sufficient to show that 
for any $n\in \mathbb{N}$,
the traces of the geometric $q^n$-th power Frobenius map $\varphi^n$ acting on $H^i(\bX,\Ql)$ and $H^i(\bar{Y},\Ql)$ are the same for each $0\leq i \leq 6$. %This condition is necessary as well: by the ``Riemann hypothesis'' portion of the Weil Conjectures, the characteristic polynomials of Frobenius acting on the $i^{\rm th}$ cohomology groups of smooth, projective varieties with equal zeta functions are equal.

Let $\varphi$ be the ($q$-th power) geometric Frobenius morphism.
By Theorem~\ref{orlov.kernel.exist},
the derived equivalence $D^b(X)\cong D^b(Y)$ is isomorphic to a Fourier--Mukai functor $\Phi_{\E }:=p_{Y*}(p_X^*(-)\otimes \E )$ for some $\E\in D^b(X\times Y)$.
Taking the traces of the action of $\varphi^*$ on
the equations \eqref{mhon},\eqref{mhtw}, 
and using the fact that the presence of a Tate twist $(j)$ has the  effect of 
multiplying the eigenvalues of the action of $\varphi^*$ on cohomology by $\frac{1}{q^{j}}$,
we have:

\begin{align}
\sum^{3}_{i=0} \frac{1}{q^{i}}\Tr(\varphi^*|H^{2i}(\bar{X},\Ql))&= \sum^{3}_{i=0} \frac{1}{q^{i}}\Tr(\varphi^*|H^{2i}(\bar{Y},\Ql)),
\label{mhtron}
\\
\sum^{3}_{i=1} \frac{1}{q^{i}}\Tr(\varphi^*|H^{2i-1}(\bar{X},\Ql))&= \sum^{3}_{i=1} \frac{1}{q^{i}}\Tr(\varphi^*|H^{2i-1}(\bar{Y},\Ql)).\label{mhtrtw}
\end{align}
The values $\Tr(\varphi^*|H^i(\bar{X},\Ql))$ and $\Tr(\varphi^*|H^i(\bar{Y},\Ql))$ are trivially equal for $i=0,6$, so~\eqref{mhtron} reduces to 
\begin{gather}\notag
\tfrac{1}{q}\Tr(\varphi^*|H^{2}(\bar{X},\Ql))
+
\tfrac{1}{q^{2}}\Tr(\varphi^*|H^{4}(\bar{X},\Ql))
\hspace*{3cm}
\\
\hspace*{3cm}
{}= \tfrac{1}{q}\Tr(\varphi^*|H^{2}(\bar{Y},\Ql))
+ \tfrac{1}{q^{2}}\Tr(\varphi^*|H^{4}(\bar{Y},\Ql)).
\label{mhtronred}
\end{gather}

By Deligne's Hard Lefschetz Theorem for $\ell$-adic \'etale cohomology \cite[Th\'eor\`eme 4.1.1]{weilii}, we have the following lemma:

\begin{lemma}[{{\cite[Lemma~4.2]{kh}}}]\label{ev}
Let $V/\F_q$ be smooth, projective variety.
If the set of eigenvalues (with multiplicity) of $\varphi^*$ acting on $H^i_{\et}(\bar{V},\Ql)$, $0\leq i<3$, are $\{\a_1,\ldots,\a_n\}$, then the set of eigenvalues of $\varphi^*$ acting on $H^{2d-i}_{\et}(\bar{V},\Ql)$ are $\{ q^{d-i}{\a_1},\ldots, q^{d-i}{\a_n} \}$.
\end{lemma}

By Lemma~\ref{ev}, \eqref{mhtronred} implies that 
$\Tr(\varphi^*|H^{2}_{\et}(\bX ,\Ql))=\Tr(\varphi^*|H^{2}_{\et}(\bar{Y},\Ql))$ and
$\Tr(\varphi^*|H^{4}_{\et}(\bX ,\Ql))=\Tr(\varphi^*|H^{4}_{\et}(\bar{Y},\Ql))$.

By Lemma~\ref{ev}, \eqref{mhtrtw} implies that 
\begin{gather}
\tfrac{2}{q}\Tr(\varphi^*|H^{1}_{\et}(\bX ,\Ql))
+
\tfrac{1}{q^2}\Tr(\varphi^*|H^{3}_{\et}(\bX ,\Ql))
\hspace*{3cm}\notag
\\
\hspace*{3cm}{}=
\tfrac{2}{q}\Tr(\varphi^*|H^{1}_{\et}(\bar{Y},\Ql))
+
\tfrac{1}{q^2}\Tr(\varphi^*|H^{3}_{\et}(\bar{Y},\Ql)).\label{last}
\end{gather}
By Corollary~\ref{claim}, 
$$\Tr(\varphi^{*}|H^{1}_{\et}(\bX ,\Ql))= \Tr(\varphi^{*}|H^{1}_{\et}(\bar{Y},\Ql)).$$ 
So, by Lemma~\ref{ev}, we have
$\Tr(\varphi^{*}|H^{5}_{\et}(\bX ,\Ql))= \Tr(\varphi^{*}|H^{5}_{\et}(\bar{Y},\Ql))$.

Since \eqref{mhon} and \eqref{mhtw} are compatible with the action of $\varphi^*$, they are compatible with the action of $\varphi^{n*}$,
and hence the above statements comparing the traces of the action of 
$\varphi^*$ also hold true if $\varphi^*$ is replaced by $\varphi^{n*}$. 
In particular,
by \eqref{last}, we have
\[
\hbox{$\Tr(\varphi^{n*}|H^{3}_{\et}(\bX ,\Ql))= \Tr(\varphi^{n*}|H^{3}_{\et}(\bar{Y},\Ql))$,}
\]
and now we have demonstrated that 
$\Tr(\varphi^{n*}|H^{i}_{\et}(\bX ,\Ql))= \Tr(\varphi^{n*}|H^{i}_{\et}(\bar{Y},\Ql))$ for all $0\leq i\leq 6$, as required.
\end{proof}

\putbib[biblio]
\end{bibunit}

\begin{bibunit}
\appendix

\section{Derived equivalent varieties have isogenous Picard varieties}
\label{app}
{\sc\small Jeffrey D. Achter, Sebastian Casalaina-Martin, Katrina Honigs and Charles Vial}

\renewcommand{\thefootnote}{}
\footnote{JDA was partially supported by  grants from the
  the NSA (H98230-14-1-0161, 
  H98230-15-1-0247 and H98230-16-1-0046).  SC-M was
  was partially supported by  a Simons Foundation
  Collaboration Grant for Mathematicians
  (317572) and NSA grant H98230-16-1-0053.  CV was supported by 
  EPSRC Early Career Fellowship
  EP/K005545/1.}

Although Popa and Schnell only claim \cite[Theorem A]{popaschnell}
that derived equivalent complex varieties have isogenous Picard
varieties, their result (and its proof) is valid, with minimal
changes, over an arbitrary field.  Our goal in this appendix is to explain:

\begin{teo}
\label{T:PSK}
Let $X$ and $Y$ be smooth projective varieties over a field $K$.  If
$X$ and $Y$ are derived equivalent, then $\pic^0(X)_\red$ and $\pic^0(Y)_\red$
are isogenous (over $K$).
\end{teo}

In outline, the proof of Theorem \ref{T:PSK} in \cite{popaschnell} for
varieties over an algebraically closed field $k$ proceeds as follows.
A theorem of Rouquier
implies that there is an isomorphism of
group schemes
\begin{equation}
\label{eqrouquierk}
F: (\aut^0_{X/k})_\red \times (\pic^0_{X/k})_\red \xra{\,\sim\,} (\aut^0_{Y/k})_\red 
\times (\pic^0_{Y/k})_\red.
\end{equation}
Unfortunately, $F$ need not preserve the given decompositions of the
source and target schemes.  Using $F$, Popa and Schnell identify
distinguished subgroups (actually, abelian varieties) $A_X \subseteq
(\aut^0_{X/k})_\red$ and $A_Y \subseteq (\aut^0_{Y/k})_\red$, 
and show
that $F$ induces an isomorphism 
\[
A_X\times (\pic^0_{X/k})_{\red}
\xra{\,\sim\,}
A_Y\times
(\pic^0_{Y/k})_{\red}.
\]
By Poincar\'e reducibility, it now suffices to show that $A_X$ and
$A_Y$ are isogenous.  For this, they construct a homomorphism 
\[
\pi:A_X\times (\pic^0_{X/k})_{\red}
\xra{\phantom{\,\sim\,}}
 A_X\times A_Y
\times\widehat A_X\times\widehat A_Y
\]
and show that  $\operatorname{Im}(\pi)$ is isogenous via the
projections $p_{13}$ and $p_{24}$ to both $A_X\times_k\widehat A_X$
and $A_Y\times_k\widehat A_Y$. 

If we now consider varieties over an arbitrary (perfect) field $K$,
since the formation of automorphism and Picard schemes commutes with
base extension, it makes sense to descend $F$ (and subsequent
constructions) from $\bar K$ to $K$.  This goes through without
incident, except that the construction of $\pi$ detailed in \cite{popaschnell}
involves a choice of point in the support of the kernel of the Mukai
transform.  We circumvent this appeal to the existence of rational
points on $X$ and $Y$ by invoking the Albanese torsor.

\subsection{Preliminaries}

For an arbitrary group scheme $G$ over a field, the maximal reduced
subscheme $G_\red$ need not be a group scheme; but this does not
happen for the Picard scheme [FGA VI.2].  For a smooth projective
variety $X/K$, let $\P(X) = \pic^0(X)_\red$ and
$\G(X) = \aut^0(X)_\red$.  Then $\P(X)$ is an abelian variety; and we
will only work with $\G(X)$ when the base field is perfect, in which
case $\G(X)$ is an irreducible group scheme.

Let $X/K$ be a geometrically reduced variety over $K$.
Then $X$
admits an (abelian) Albanese variety $\alb(X)/K$, a torsor
$\alb^1(X)$, and a morphism $X \to \alb^1(X)$ which is universal for
morphisms from $X$ into torsors under abelian varieties (see, e.g., \cite[\S
2]{wittenbergalbanese}).

If $P \in X(K)$ is a point, then there is a pointed morphism $(X,P)
\to (\alb(X), \calo)$ which is universal for pointed morphisms from
$X$ to abelian varieties.  We will sometimes denote $\alb(X)$,
together with this morphism, as  $\alb(X,P)$. 
% If $P$ and $Q$ are
%two different points of $X$, then there is a canonical isomorphism of
%abelian varieties $\alb(X,P) \to \alb(X,Q)$.

Let $\G(X) = \aut^0(X)_\red$.  If $P \in X(K)$ is a base point, Popa
and Schnell compute \cite[Lemma 2.2]{popaschnell} a canonical morphism
\[
\xymatrix{
\alb(\G(X)) \ar[r]^{f_{(X,P)}} & \alb(X,P).
}
\]
\goodbreak

\begin{lem}
\label{L:albacts}
Let $X/K$ be a variety over a perfect field.
\begin{alphabetize}
\item There is a canonical action $\alb(\G(X)) \times \alb^1(X) \to \alb^1(X)$
  which induces a canonical morphism 
\[
\xymatrix{
\alb(\G(X)) \ar[r]^{g_X} &\alb(X).
}
\]
\item If $P \in X(K)$ is a point, then the trivialization
  $\alb(X,P) \stackrel{\sim}{\to} \alb^1(X)$ makes the following diagram commute:
\[
\xymatrix{
\alb(\G(X)) \ar[d]^= \ar[r]^{f_{(X,P)}} & \alb(X,P) \ar[d]^\sim\\
\alb(\G(X)) \ar[r]_{g_X} & \alb^1(X).
}
\]
\end{alphabetize}
\end{lem}

\begin{proof}
By the universal property of the Albanese, the composition $\G(X)
\times X \to X \to \alb^1(X)$ factors through $\alb^1(\G(X)\times X)
\iso \alb^1(\G(X)) \times \alb^1(X)$.  Since $\G(X)$ admits a
$K$-rational point, its Albanese torsor coincides with its Albanese
variety, and we obtain an action
\[
\xymatrix{
\alb(\G(X)) \times \alb^1(X) \ar[r] & \alb^1(X).
}
\]
In particular, $\alb(\G(X))$ acts as a connected group scheme of
automorphisms of $\alb^1(X)$; by Lemma \ref{L:auttorsor}, we obtain a
morphism $g_X:\alb(\G(X)) \to  \alb(X)$.  

Popa and Schnell construct $f_{(X,P)}$ in a similar way, except that
they work with the Albanese varieties of the pointed varieties
$(\G(X),\id)$ and $(X,P)$.  Part (b) then follows from the
universality of $X \to \alb^1(X)$ into 
abelian torsors and of $X \to \alb(X,P)$ into abelian varieties.
\end{proof}

\begin{lem}
\label{L:auttorsor}
Let $A/K$ be an abelian variety, and let $T/K$ be a torsor under $A$.
Then $\aut(T)^0 \iso A$.
\end{lem}

\begin{proof}
Over an algebraic closure, we have $\aut(T)_{\bar
K}^0 \iso \aut(A)_{\bar K}^0 \iso A_{\bar K}$.  Therefore, the
inclusion $A\hookrightarrow \aut(T)^0$ induced by the faithful action
of $A$ on $T$ is an isomorphism.
\end{proof}

\subsection{Proof of the Popa--Schnell theorem}

Having dispatched these preliminaries, we now explain how to adapt the
proof of the complex version of Theorem \ref{T:PSK} in
\cite{popaschnell} to account for an arbitrary base field.  At each stage,
we will see that the morphisms used in \cite{popaschnell}, a priori
defined over an algebraically closed field, actually descend to a
field of definition.

Let $\bar K$ be an algebraic closure of $K$; for a variety $Z/K$, let
$\bar Z = Z_{\bar K}$.

\begin{proof}[Proof of Theorem \ref{T:PSK}]
By Chow rigidity \cite[Thm.\ 3.19]{conradtrace}, two abelian varieties
over $K$ are isogenous if and only if they are isogenous over the
perfect closure of $K$.  Consequently, to prove the theorem we may and
do assume 
that $K$ is perfect.  Note then that if $G$ and $H$ are
group schemes over $K$, then $G_\red$ is again a group scheme, and
$(G\times H)_\red \iso G_\red \times H_\red$.  
Moreover, we have
$\G(\bar Z) \iso \G(Z)_{\bar K}$ and $\P(\bar Z) = \P(Z)_{\bar K}$.

Let $\Phi:\derived^b(X) \to \derived^b(Y)$  be an equivalence of
categories.  A fundamental theorem of Orlov
(Theorem \ref{orlov.kernel.exist})
 asserts that there is an object $\mathcal E
\in \derived^b(X\times Y)$   such that $\Phi$ is given by 
\[
\xymatrix{
\Phi = \Phi_{\mathcal E}: \mathcal M\udot \ar@{|->}[r] &  p_{Y*}(p_X^*
\mathcal M\udot
\tensor \mathcal E)}.
\]
Over $\bar K$, a theorem of Rouquier \cite[Thm. 4.18]{rouquier11} shows that $\Phi$ induces an
isomorphism
\begin{equation*}
%\label{E:rouquier}
\xymatrix{
\aut^0(\bar X) \times \pic^0(\bar X) \ar[r] & \aut^0(\bar Y)
\times \pic^0(\bar Y).
}
\end{equation*}
This induces an isomorphism on reduced subschemes which we
denote $\bar F$:
\begin{equation}
\label{E:rouquierred}
\xymatrix{
\G(X)_{\bar K} \times \P(X)_{\bar K}\ar[r]^{\bar F} & \G(Y)_{\bar K} \times
\P(Y)_{\bar K}
}
\end{equation}
(Note that $\G(X)(\bar K) = \aut^0(X)(\bar K)$, etc.)
On points, $\bar F$ is characterized by the fact that
\[
\bar F(\phi, \mathcal L) = (\psi, \mathcal M) \Longleftrightarrow
p_X^*\mathcal L \otimes (\phi\times \id)^*\mathcal E \iso
p_Y^*\mathcal M \otimes (\psi\times \id)_*\mathcal E.
\]
Since $\mathcal E$ is defined over $K$, the graph of this relation in $\G(X)_{\bar K} \times \P(X)_{\bar K}
\times \G(X)_{\bar K} \times \P(X)_{\bar K}$ is stable under
$\aut(\bar K/K)$,  and so 
isomorphism \eqref{E:rouquierred} descends to an isomorphism
\[
\xymatrix{
\G(X) \times \P(X)\ar[r]^{F} & \G(Y) \times
\P(Y)
}
\]
of connected, reduced group schemes over $K$.

Using the projections $p_{\G(Y)}$ and $p_{\G(X)}$,  we obtain $K$-rational morphisms
\begin{align*}
\P(X) &\xra{
\rlap{\scriptsize $\alpha_Y= p_{\G(Y)}\circ F$} 
\phantom{\alpha_X= p_{\G(X)} \circ F\inv\,}
}  
\G(Y) \\
\P(Y) &\xra{\alpha_X= p_{\G(X)} \circ F\inv\,}  \G(X);
\end{align*}
let $A_X = \alpha_X(\P(Y)) \subseteq \G(X)$ and $A_Y = \alpha_Y(\P(X)) \subseteq
\G(Y)$.  (Note that, since the formation of kernels commutes with base
change, $\bar A_X \iso A_{\bar X}$.)  The pointwise argument of \cite{popaschnell}, combined with
the fact that $F$ admits an inverse, shows that $F$ induces an
isomorphism
\begin{equation*}\label{F}
A_X \times \P(X) \lra A_Y \times \P(Y)
\end{equation*}
of abelian varieties over $K$.
By Poincar\'e reducibility, it now suffices to show that $A_X$ and $A_Y$ are isogenous.

Over an algebraically closed field, Popa and Schnell choose a point $(P,Q) \in (X\times Y)(\bar
K)$ in the support of $\mathcal E$, and use it to define morphisms of
varieties over $\bar K$:
\begin{gather*}
\bar A_X \times \bar A_Y \xrightarrow{\,\bar f = \bar f_X \times \bar f_Y\,}
\bar X \times \bar Y 
\\
\hspace*{1cm}(\phi,\psi) \longmapsto(\phi(P),\psi(Q))
\end{gather*}
The dual map $\bar f^*: \P(\bar X) \times
\P(\bar Y) \to \widehat{\bar{A_X}} \times \widehat{\bar{A_Y}}$ is surjective.

Working now over a field which is only assumed to be perfect, Lemma \ref{L:albacts} supplies a canonical morphism $g_X:
\alb(\G(X)) \to \alb^1(X)$ whose
base change to $\bar K$ is $g_{X, \bar K} \iso \bar f_X$.  In
particular, by \cite[Lemma
2.2]{popaschnell}, which depends only on \cite{brion10} and is valid
in any characteristic, 
$H_X := \ker g_X$ is a finite group scheme.
Similarly, there is a canonical morphism $g_Y: \alb(\G(Y)) \to \alb(Y)$ with
finite kernel $H_Y$, and $g_{Y,\bar K} \iso \bar f_Y$.  We obtain a surjection $g^*: \P(X) \times \P(Y)
\to \widehat A_X \times \widehat A_Y$ of abelian varieties over $K$
which is an isogeny onto its image.

Consider the morphism
\[
A_X \times  A_Y \xra{\,\tau = (\tau_1, \tau_2  \pi_2)\,\,}
(A_X\times A_Y) \times (\widehat A_X \times \widehat A_Y)
\]
of abelian varieties over $K$, where
\begin{align*}
\tau_1 &= (\id_{A_X}, p_{\G(Y)} \circ F), \\
\tau_2 &= (g_X^*\circ \iota, g_Y^*\circ p_{\P(Y)} \circ F),
\end{align*}
and $\iota$ denotes the inversion map on the abelian variety
$\widehat{A_X}$. Let $p_{13}$ (respectively, $p_{24}$) denote the
projection of the codomain of $\tau$ onto the first and third
(respectively, second and fourth) components.

After base change to $\bar K$, the morphisms $\bar \tau_1$, $\bar
\tau_2$ and $\bar \tau$ coincide, respectively,  with the morphisms $\pi_1$, $\pi_2$
and $\pi$ constructed in \cite[p.533]{popaschnell}.  In
particular $\bar{p_{13}\circ \tau}$ and $\bar{p_{24}\circ\tau}$, and
therefore $p_{13}\circ \tau$ and $p_{24}\circ \tau$, are isogenies.
By Poincar\'e reducibility, $A_X$ and $A_Y$ are isogenous.
\end{proof}

\begin{cor}\label{claim}
Let $X$ and $Y$ be smooth projective varieties over a field $K$.  If
$X$ and $Y$ are derived equivalent, then $H^1(X_{\bar K},\rat_\ell)
\iso H^1(Y_{\bar K},\rat_\ell)$ and $H^{2d-1}(X_{\bar K},\rat_\ell)
\iso H^{2d-1}(Y_{\bar K},\rat_\ell)$ as representations of $\gal(\bar K/K)$, where $d = \dim X = \dim Y$ and $\ell$ is invertible in $K$.
\end{cor}

\begin{proof}
The claim for cohomology in degree one follows from Theorem
\ref{T:PSK} and the canonical identifications
$\pic^0(X)[\ell^n](\bar K) \iso H^1(X_{\bar K}, {
  \mu}_{\ell^n})$ provided by the Kummer sequence.
The second claim now follows from Poincar\'e duality.
\end{proof}

\putbib
\end{bibunit}

\end{document}